\documentclass[
12pt]{article}

\usepackage[colorlinks=true, pdfstartview=FitV, linkcolor=blue, citecolor=blue, urlcolor=blue]{hyperref}
\usepackage{amsmath,amsfonts,amssymb}
\usepackage{times, verbatim}

\DeclareFontFamily{OT1}{rsfs}{}
\DeclareFontShape{OT1}{rsfs}{n}{it}{<-> rsfs10}{}
\DeclareMathAlphabet{\mathscr}{OT1}{rsfs}{n}{it}

\newtheorem{thm}{Theorem}
\newtheorem{cor}[thm]{Corollary}
\newtheorem{lem}[thm]{Lemma}
\newtheorem{prop}[thm]{Proposition}
\newenvironment{proof}{\noindent {\bf Proof:}}{$\Box$ \vspace{2 ex}}

\def\Z{{\mathbb Z}}
\def\Q{{\mathbb Q}}
\def\R{{\mathbb R}}

\def\P{{\mathbb P}}
\def\F{{\mathbb F}}
\DeclareMathOperator{\nr}{nr}

\DeclareMathOperator{\GL}{GL}

\DeclareMathOperator{\PGL}{PGL}

\def\<#1>{\left<#1\right>}

\newenvironment{ProofOf}[1]{\par\noindent{\bf Proof of #1:}}%
                       {\hspace*{\fill}\nobreak$\Box$\par\medskip}

\title{The proportion of plane cubic curves
  over $\Q$ that everywhere locally have a point}
\author{Manjul Bhargava, John Cremona, and Tom Fisher}

\begin{document}

\maketitle

\begin{abstract}
  We show that the proportion of plane cubic curves over $\Q_p$ that
  have a $\Q_p$-rational point is a rational function in $p$, where
  the rational function is independent of $p$, and we determine this rational function
  explicitly.  As a consequence, we
  obtain the density of plane cubic curves over~$\Q$ that have
  points everywhere locally; numerically, this density is shown to be
  $\approx 97.3\%$.
\end{abstract}

\section{Introduction} 
\label{sec:intro}

Any plane cubic curve over $\Q$ may be defined by the vanishing of a
ternary cubic form
\begin{equation}\label{Cdef}
  C(X,Y,Z) = aX^3+bX^2Y+cX^2Z+dXY^2+eXYZ+fXZ^2+gY^3+hY^2Z+iYZ^2+jZ^3
\end{equation}
where all coefficients $a,\ldots,j$ lie in $\Z$.  We say that a
ternary cubic form $C$ is {\it everywhere locally soluble} if 
it has a nontrivial
zero over every completion of $\Q$, i.e., if the corresponding plane
cubic curve has a point everywhere locally.  In this paper, we wish to
determine the probability that a random such integral ternary cubic
form is everywhere locally soluble.

More precisely, define the height $h(C)$ of~the cubic form $C$
in~(\ref{Cdef}) by $h(C):= \max\{|a|,\ldots,|j|\}.$ Then, as in the
work of Poonen and Voloch~\cite{PV}, we define the probability that a
random plane cubic curve over $\Q$ has a point everywhere locally
(equivalently, the probability that a random integral ternary cubic
form is everywhere locally soluble) by
\begin{equation*} %\label{probdef}
\rho = \lim_{B\to\infty} \frac{\#\{C(X,Y,Z): C \mbox{ is everywhere locally
    soluble and } h(C)<B  \}}
{\#\{C(X,Y,Z): h(C)<B  \} }.
\end{equation*}
It is proven in~\cite[Thm.\ 3.6]{PV}, using the sieve of
Ekedahl~\cite{Ek}, that this limit exists and is given by
\[ \rho=\prod_p \rho(p), \] 
where the product is over all primes $p$; here $\rho(p)$ denotes the
probability (with respect to the usual additive
$\Z_p$-measure) that a random ternary cubic form over $\Z_p$ is {\it soluble} over $\Q_p$, i.e., has a
nontrivial zero over~$\Q_p$. There is no contribution from the infinite place
because a plane cubic curve over $\R$ always has real points.

The purpose of this paper is to determine the probabilities $\rho(p)$
for all primes $p$, and thus $\rho$, explicitly.  Specifically, we will prove
that $\rho(p)$
is a rational function in $p$, where the rational function is
independent of $p$:
\begin{thm}\label{mainlocal}
The probability that a random plane cubic curve over $\Q_p$ has a
$\Q_p$-rational point is given by
% \[ \rho(p) = 1 - \frac{p^9 - p^8 + p^6 - p^4 + p^3 + p^2 - 2 p + 1}
% {3 (p^2+1) (p^4+1) (p^6+p^3+1)}. \]
\[ \rho(p) = 1-\frac{f(p)}{g(p)} \]
where $f$ and $g$ are the following integer coefficient polynomials of 
degrees~$9$ and~$12$:
% degrees~$26$ and~$29$:
\begin{align*}
f(p) &= p^9 - p^8 + p^6 - p^4 + p^3 + p^2 - 2 p + 1, \\[.05in]
g(p) &= 3 (p^2+1) (p^4+1) (p^6+p^3+1).
\end{align*}
\end{thm}

Note that
\[
    1-\rho(p) =\frac{f(p)}{g(p)} 
                  \sim \frac{1}{3p^3},
\]
so $\rho(p)\to1$ rapidly as $p\to\infty$; for small~$p$, we have 
$\rho(2) \approx 0.98319$, $\rho(3) \approx 0.99259$, $\rho(5) \approx 
0.99795$, and $\rho(7) \approx 0.99918$.  

% > 1.0*Evaluate(prob_sol,2);
% 0.983185603008326618318560300833
% > 1.0*Evaluate(prob_sol,3);
% 0.992588952970110942853153762713
% > 1.0*Evaluate(prob_sol,5);
% 0.997948690851753145277589082871
% > 1.0*Evaluate(prob_sol,7);
% 0.999183672460809810290448485493
% > 1.0*Evaluate(prob_sol,11);
% 0.999774194993519090058357194921

%\begin{align*}
%\rho(2) &= 0.983125226721428\\
%\rho(3) &= 0.992584816271550\\
%\rho(5) &= 0.997948609966566\\
%\rho(7) &= 0.999183667463936
%\end{align*}

{}From Theorem~\ref{mainlocal}, we conclude:
\begin{thm}\label{main}
The probability that a random plane cubic curve over $\Q$ has a point
locally everywhere is given by
\[\rho =\prod_p \rho(p) = \prod_p \left(1-\frac{f(p)}{g(p)}\right).\]
\end{thm}
Numerically, we have 
%$\rho \approx 0.9725608$.
$\rho \approx 97.256\%$.

%; thus the probability that a random plane cubic curve over $\Q$ has a point everywhere locally is quite high.

The feature that $\rho(p)$ is a rational function of $p$ independent
of $p$ is special and does not always occur in other contexts.  For
example, for the genus one models $y^2=f(x,z)$, where $f$ is a binary
quartic form over $\Z$ (or more generally, a binary form of degree
$2g+2$ yielding a hyperelliptic curve of genus $g$), we show in
\cite{BCF3} that the analogue of $\rho(p)$ cannot be any fixed
rational function of $p$.

Our method for proving Theorem~\ref{mainlocal} is based on that for testing
solubility of a smooth plane cubic over $\Z_p$ (equivalently, $\Q_p$) as described, for
example, in \cite[Section 2]{FisherSills}; the arguments are
 also reminiscent of those used for minimising ternary cubics, as in
\cite{CFS}, and of those used to determine the density of isotropic integral
quadratic forms, as in \cite{BCF}.  Namely,
we consider the reductions modulo~$p$; cubics whose reductions have
smooth $\F_p$-points are soluble by Hensel's Lemma, while those that have
no $\F_p$-points are insoluble.  To determine the probabilities of solubility in the more difficult remaining cases, we develop certain recursive formulae, involving these and other suitable related quantities, that allow us to solve and obtain exact algebraic expressions for the desired probabilities.

We remark that the analogue of Theorem~\ref{mainlocal} holds (with the same proof)
over any finite extension of $\Q_p$: we simply replace $p$ by a uniformiser
when making substitutions in the proofs, and replace $p$ by the order
of the residue field when computing probabilities. The analogue of Theorem~\ref{main} then also holds with any number field in place of $\Q$ (where we define the relevant probability following \cite[\S4]{PV}).

This paper is organized as follows.  
In Section~\ref{sec:notpre} we provide counts for some polynomials and forms over 
finite fields that will be required in the proof of Theorem~\ref{mainlocal}. We then prove Theorem~\ref{mainlocal} in
Section~\ref{sec:thm12pf}, using the recursive methods described above.  Finally, in Section~\ref{sec:unr}, we give an extension of Theorem~\ref{mainlocal} where we determine the probability that a random plane cubic curve over $\Q_p$ has
a point {over {some} unramified extension of $\Q_p$}. We find that the answer takes a particularly
simple form in this case, and we outline the necessary changes required for the proof.

\section{Some counting over finite fields} \label{sec:notpre}

We work over the finite field $\F_q$ with $q$ elements, where $q$ is a
prime power.

\subsection{Cubic polynomials and binary cubic forms over $\F_q$.}

\begin{lem}\label{lem:cubic-count}
Of the $q^3$ monic cubics $g\in\F_q[X]$,
\begin{itemize}
\item $q^2(q-1)$ have distinct roots, of which
\begin{itemize}
\item[$\cdot$] $\frac{1}{6}q(q-1)(q-2)$ have distinct roots in~$\F_q$;
\item[$\cdot$] $\frac{1}{2}q^2(q-1)$ have one root in~$\F_q$ and two
  conjugate roots in~$\F_{q^2}$;
\item[$\cdot$] $\frac{1}{3}q(q^2-1)$ have three conjugate roots in~$\F_{q^3}$;
\end{itemize}
\item $q(q-1)$ have a simple root and a double root (both in $\F_q$);
\item $q$ have a triple root (in $\F_q$);
\end{itemize}
\end{lem}

\begin{cor}\label{cor:sigma1tau1}
  The probability that a random monic cubic over~$\F_q$ has a simple
  root in~$\F_q$ is $\sigma_1 = \frac{2}{3}(q^2-1)/q^2$, and the
  probability that it has a triple root is $\tau_1=1/q^2$.
\end{cor}

\begin{lem}\label{lem:hom-cubic-count}
  Of the $q^4$ binary cubic forms $g\in\F_q[X,Y]$,
\begin{itemize}
\item $q(q^2-1)(q-1)$ have distinct roots, of which
\begin{itemize}
\item[$\cdot$] $\frac{1}{6}q(q^2-1)(q-1)$ have all three roots
  in~$\F_q$;
\item[$\cdot$] $\frac{1}{2}q(q^2-1)(q-1)$ have one root in~$\F_q$ and
  two conjugate roots in $\F_{q^2}$;
\item[$\cdot$] $\frac{1}{3}q(q^2-1)(q-1)$ have three conjugate roots
  in~$\F_{q^3}$;
\end{itemize}
\item $q(q^2-1)$ have a simple and a double root (both in $\F_q$);
\item $q^2-1$ have a triple root (in $\F_q$);
\item $1$ is the zero form. % is~$0$.
\end{itemize}
\end{lem}

\begin{cor}\label{cor:sigmatau}
  The probability that a random binary cubic form over~$\F_q$ has a
  simple root in~$\P^1(\F_q)$ is $\sigma =
  \frac{1}{3}(q^2-1)(2q+1)/q^3$, and the probability that it has a
  triple root is $\tau=(q^2-1)/q^4$.
\end{cor}

\subsection{Ternary cubic forms over $\F_q$.}

We count the numbers of ternary cubic forms over~$\F_q$, separating
out all the different possible cases: smooth, irreducible but
singular, or reducible in various ways.

We first count forms up to scaling by elements of $\F_q^\times$; these
counts are then multiplied by $q-1$, and after adding~$1$ for the
zero form, the counts for cubics add up to~$q^{10}$.

\subsubsection{Lines} The number of lines in $\P^2$ over~$\F_{q^k}$
is~$n_1^{(k)}=(q^{3k}-1)/(q^k-1)$.  We write $n_1=n_1^{(1)}=q^2+q+1$.

\subsubsection{Conics} The total number of conics~$Q$ over~$\F_q$
is~$(q^6-1)/(q-1)$.  Reducible conics are of three types, up to
scaling: $Q=L^2$ with $L$ a line over~$\F_q$, or $Q=L_1L_2$ with lines
$L_j$ either both defined over~$\F_q$ or conjugate over $\F_{q^2}$.
The counts for these types are, up to scaling;
\begin{itemize}
\item $\#\{Q=L^2\}= n_1 = q^2+q+1$.
\item $\#\{Q=L_1L_2\;\text{  over~$\F_q$}\} = n_1(n_1-1)/2 =
  \frac{1}{2}q(q+1)(q^2+q+1)$.
\item $\#\{Q=L_1L_2\;\text{  conjugate over~$\F_{q^2}$}\} =
  (n_1^{(2)}-n_1)/2 = \frac{1}{2}q(q-1)(q^2+q+1)$.
\end{itemize}
Summing, we find that the number of absolutely reducible conics
is~$(q^2+1)(q^2+q+1)$.  Hence the number of absolutely irreducible
conics is~$q^5-q^2 = q^2(q-1)(q^2+q+1)$, up to scaling.

% \subsubsection{Reducible plane cubic curves}
\subsubsection{Plane cubic curves}
The number of plane cubic curves over~$\F_q$, up to scaling,
is~$(q^{10}-1)/(q-1)$.

Reducible cubics are of the form $C=L^3$, $C=L_1^2L_2$,
or~$C=L_1L_2L_3$ (with all lines~$L_j$ defined over $\F_q$) or
$C=L_1L_2L_3$ (three lines conjugate over $\F_{q^3}$), or $C=LL_1L_2$
(with $L$ over $\F_{q}$ and $L_1,L_2$ conjugate over $\F_{q^2}$), or
$C=LQ$ with $Q$ an absolutely irreducible conic.
%quadratic.  
The counts for these are, up to scaling:
\begin{itemize}
\item $\#\{C=L^3\} = n_1  = q^2+q+1$.
\item $\#\{C=L_1L_2^2\} = n_1(n_1-1) = q(q+1)(q^2+q+1)$.
\item $\#\{C=L_1L_2L_3\;\text{  over~$\F_q$}\} = n_1(n_1-1)(n_1-2)/6
  = \frac{1}{6}q(q+1)(q^2+q-1)(q^2+q+1)$.
\item $\#\{C=L_1L_2L_3\;\text{  conjugate over~$\F_{q^3}$}\} =
  (n_1^{(3)}-n_1)/3 = \frac{1}{3}q(q^2-1)(q^3+q+1)$.

  Of these, we will later need to know the number that are concurrent
  (forming a ``star'') and the number that are not (forming a
  ``triangle''), which are 
\begin{itemize}
\item (star) $\frac{1}{3}q(q^2-1)(q^2+q+1)$;
\item (triangle) $\frac{1}{3}q^3(q^2-1)(q-1)$.
\end{itemize}
\item $\#\{C=LL_1L_2\;\text{  conjugate over~$\F_{q^2}$}\} =
  n_1(n_1^{(2)}-n_1)/2 = \frac{1}{2}q(q-1)(q^2+q+1)^2$.
\item $\#\{C=LQ\} = n_1(q^5-q^2) = q^2(q-1)(q^2+q+1)^2$.
\end{itemize}
Adding, we find that the number of absolutely reducible cubics is $(q
+ 1) (q^6 + q^5 + q^4 + q^2 + 1)$, and hence the number of absolutely
irreducible cubics is
\[
   q^5(q+1)(q-1)(q^2 + q + 1) = N q^2/(q-1)
\]
where $N = |\PGL(3,\F_q)|$.  Although it will not be needed for the
proof of Theorem~\ref{mainlocal}, we remark that $Nq$ of these are
smooth, with exactly $N$ for each of the $q$ possible $j$-invariants. Of
the remaining cubics, $N$ have a node and $N/(q-1)$ have a cusp.

\section{Proof of Theorem~\ref{mainlocal}}\label{sec:thm12pf}

In this section, we prove Theorem~\ref{mainlocal} giving the density
of plane cubic curves over~$\Q_p$ that have a $\Q_p$-rational point.

\subsection{Outline of the proof}

Let $C$ be a ternary cubic form with coefficients in~$\Z_p$ that is
{\it primitive}, meaning that not all of its coefficients are divisible
by~$p$.  We say that $C$ is {\em soluble} if there exist
$x,y,z\in\Q_p$, not all zero, such that $C(x,y,z)=0$, or in other
words, if the associated cubic curve has a $\Q_p$-rational point.

The reduction of~$C$ modulo~$p$ is a cubic curve~$\overline{C}$
over~$\F_p$.  If~$\overline{C}$ has no~$\F_p$-rational points then
certainly $C$ is not soluble.  Since lines, smooth conics, and
absolutely irreducible cubics over~$\F_p$ always have smooth points,
inspection of the cases enumerated in the previous section shows
that~$\overline{C}$ always has smooth $\F_p$-points (even when~$p=2$)
unless it is a product of three lines conjugate over~$\F_{p^3}$
or a triple line; and in the former case, if the three lines are not
concurrent then there are no $\F_p$-rational points at all.

It follows that $C$ is not soluble when the reduction~$\overline{C}$
consists of three non-concurrent lines conjugate over~$\F_{p^3}$ (the
{\it triangle} configuration); it may or may not be soluble
when~$\overline{C}$ is three concurrent lines conjugate
over~$\F_{p^3}$ (the {\it star} configuration) or a triple line; and in
all other cases $C$ is soluble.  It remains to determine the density
of soluble cubics whose reduction is either a star or a triple
line.  The density will not depend on which $\F_p$-rational point is
the common point in the first case, or which $\F_p$-line is the triple
line in the second, so in Sections~\ref{sec:alpha1}
and~\ref{sec:alpha2} we take them to be $P_0=[1:0:0]$ and $X=0$,
respectively.  We will then develop recursions in order to solve for the densities
of soluble cubics among those whose reductions lie in the star and triple line configurations.

\subsection{Preliminaries}

We now compute the probabilities of certain configurations occurring,
including the ones already mentioned and two others which will arise
in the course of the proof.

Let $\beta_1$ denote the probability of a ternary cubic over $\Z_p$ having a
reduction which is a star as defined above, $\beta_2$ the probability
of the reduction being a triple line, and $\beta_3$ the probability of
the reduction being a triangle, again as defined above.  Additionally,
let $\beta_4$ be the probability of the {\it line condition}, defined to
be the property that the reduction of the cubic meets the line $X=0$ in three
distinct points conjugate over~$\F_{p^3}$, and let $\beta_5$ be the
probability of the {\it point condition}, defined to be the property
that the reduction of the cubic does not contain the point $P_0=[1:0:0]$.

We will need to know the density of cubics over $\Z_p$ satisfying each of these
conditions, as well as the relative density of those satisfying the
star, triple line, and triangle conditions among those satisfying each
of the line and point conditions.  These relative densities will be
denoted by $\beta_j'$ and $\beta_j''$, respectively, for $j=1,2,3$.

\begin{prop} \label{prop:betas} The probabilities of a random plane
  cubic over~$\F_p$ satisfying each of these five conditions are as
  follows:
\begin{enumerate}
\item (star: all; relative to line condition; relative to point
  condition)
\[ \beta_1  = \frac{(p^2-1)(p^3-1)}{3p^9};\qquad
   \beta_1' = \frac{1}{p^4};\qquad
   \beta_1'' = \frac{(p+1)^2(p-1)}{3p^7};
\]
\item (triple line: all; relative to line condition; relative to point
  condition)
\[
   \beta_2 = \frac{p^3-1}{p^{10}};\qquad
   \beta_2'= 0;\qquad
   \beta_2'' = \frac{1}{p^{7}};
\]
\item (triangle: all; relative to line condition; relative to point
  condition)
\[
   \beta_3 = \frac{(p+1)(p-1)^3}{3p^7};\qquad
   \beta_3'= \frac{p-1}{p^4};\qquad
   \beta_3'' = \frac{(p+1)(p-1)^2}{3p^6};
\]
\item (line condition)
\[
   \beta_4 =  \frac{(p+1)(p-1)^2}{3p^3};
\]
\item (point condition)
\[
  \beta_5 = \frac{p-1}{p}.
\]
\end{enumerate}
\end{prop}

\begin{proof}
  We refer to the previous section for the numbers of stars, triple
  lines and triangles, up to scaling.  Multiplying by $(p-1)/p^{10}$
  gives $\beta_j$ for $j=1,2,3$.
\begin{enumerate}
\item The probability of satisfying the star condition centred at a
  given point in $\P^2(\F_q)$ is
  \[ \gamma = \frac{1}{3}(p^3-p) \cdot \frac{p-1}{p^{10}}. \] We have
  already computed $\beta_1 = (p^2 + p + 1) \gamma$. The probability
  of satisfying both the line and star condition is $p^2 \gamma$.
  Dividing by~$\beta_4$ gives $\beta_1'$.  Similarly, the probability
  of satisfying both the point and star condition is $(p^2 + p)
  \gamma$
  which on dividing by~$\beta_5$ gives $\beta_1''$.

\item The triple line and line conditions cannot occur together, so
  $\beta_2'=0$.  The triple line and point conditions together have
  probability $p^2(p-1)/p^{10}$, and dividing by~$\beta_5$ gives
  $\beta_2''$.

\item Since the triangle condition implies both the line and point
  conditions, $\beta_3'=\beta_3/\beta_4 = (p-1)/p^4$, and
  $\beta_3''=\beta_3/\beta_5={(p+1)(p-1)^2}/{3p^6}$.

\item The line condition holds with the same probability that a binary
  cubic form over~$\Z_p$ is irreducible modulo~$p$, which is
\[
   \beta_4 = \frac{1}{3}(p^3-p) \cdot \frac{p-1}{p^4}.
\]

\item The point condition is equivalent to the condition that the
  coefficient of~$X^3$ in the cubic form is not divisible by~$p$, so
  occurs with probability~$\beta_5 = 1-1/p$.
\end{enumerate}
\end{proof}

For $1\le j\le5$, let $\alpha_j$ denote the probability of solubility
for ternary cubics whose reduction is a star, a triple line, a
triangle, or which satisfy the line or point conditions, respectively.
We know that $\alpha_3=0$.  In the next two subsections, we compute
$\alpha_1$, $\alpha_4$, $\alpha_2$, and $\alpha_5$.

\subsection{The star case: computation of $\alpha_1$ (together with $\alpha_4$)}
\label{sec:alpha1}

We derive two linear equations linking~$\alpha_1$ and~$\alpha_4$, from
which their values may be determined.

\begin{lem}\label{lem:lem8}
  $1-\alpha_4 = \beta_1'(1-\alpha_1)+\beta_2'(1-\alpha_2)+\beta_3'$.
  Hence $\alpha_4 = (p^4-p+\alpha_1)/p^4$.
\end{lem}
\begin{proof}
  For the first equation, we combine the probablities of insolubility
  in the star, triple line, and triangle cases, the latter being~$1$.
  Using the values of~$\beta_j'$ from Proposition~\ref{prop:betas}
  allows us to solve for~$\alpha_4$ in terms of $\alpha_1$ alone,
  since the coefficient of~$\alpha_2$ is conveniently $\beta_2'=0$.
\end{proof}

\begin{lem}
       $\alpha_1 = (p^3-p+\alpha_4)/p^4$.
\end{lem}
\begin{proof}
  Without loss of generality the centre of the star is at
  $P_0=[1:0:0]$, so the cubic has the form
\[
   C = c_0X^3 + c_1(Y,Z)X^2 + c_2(Y,Z)X + c_3(Y,Z)
\]
where each $c_j$ is a binary form of degree~$j$ and the
valuations\footnote{By the valuation of a form we mean the minimal
  valuation of its coefficients.} of the $c_j$ are
\[
     {}\ge1 \qquad {}\ge1 \qquad {}\ge1 \qquad {}=0
\]
with $c_3$ irreducible modulo~$p$.  Solubility of~$C$ means that there
exist~$x,y,z\in\Z_p$, not all in~$p\Z_p$, satisfying~$C(x,y,z)=0$.
Here, for any such solution, irreducibility of~$c_3$ modulo~$p$
implies that $y,z\equiv0\pmod{p}$.

If $v(c_0)=1$ (which has probability $1-1/p$), then $C$ is insoluble
since the first term has valuation~$1$ while the other terms have
valuation at least~$2$.  So we may assume that $v(c_0)\ge2$ (which has
probability $1/p$), substitute $pY, pZ$ for $Y,Z$, and divide by~$p^2$.
The valuations of the terms are now
\[
     {}\ge0 \qquad {}\ge0 \qquad {}\ge1 \qquad {}=1.
\]

If $v(c_1)=0$ (which has probability $1-1/p^2$ since~$c_1$ has two
coefficients), then $C\pmod{p}$ has a simple linear factor over~$\F_p$
and hence is soluble.  Otherwise, we have $v(c_1)\ge1$ (which has probability
$1/p^2$), and the valuations of the terms satisfy
\[
     {}\ge0 \qquad {}\ge1 \qquad {}\ge1 \qquad {}=1.
\]

We must have $x\not\equiv0\pmod{p}$ by primitivity, since we have
already forced $y\equiv z\equiv 0$ (in the original coordinates).
Hence for solubility we must have $c_0\equiv0\pmod{p}$; that is,
$v(c_0)\ge1$.  Assuming this, which has probability~$1/p$, we may
divide through by~$p$ to obtain valuations
\[
     {}\ge0 \qquad {}\ge0 \qquad {}\ge0 \qquad {}=0.
\]
Recalling that $c_3$ is irreducible modulo~$p$, we see that this is an
arbitrary cubic satisfying the line condition, so the probability of
solubility is~$\alpha_4$.

Tracing through the above steps, we see that
\[
  \alpha_1 = (1-1/p)\cdot0 +
  (1/p)\cdot\left((1-1/p^2)\cdot1 + (1/p^2)\cdot((1-1/p)\cdot0 +
  (1/p)\cdot\alpha_4)\right),
\]
which simplifies to the equation stated.
\end{proof}

Solving for~$\alpha_1$ now gives
\begin{prop}
\[
   \alpha_1 = \frac{p^7-p^5+p^4-p}{p^8-1}
            = \frac{p(p-1)(p^5 + p^4 + p^2 + p + 1)}{p^8-1}.
\]
\end{prop}

\subsection{The triple line case: computation of $\alpha_2$ (together with $\alpha_5$)}
\label{sec:alpha2}

Recall that $\alpha_2$ and $\alpha_5$ denote the probabilities of
solubility given the star and point conditions, respectively.  We derive
two equations linking these quantities (and $\alpha_1$), from which
they may be determined.

First, we have the analogue of Lemma~\ref{lem:lem8}:
\begin{lem}\label{lem:eq1}
$1-\alpha_5 = \beta_1''(1-\alpha_1) + \beta_2''(1-\alpha_2) + \beta_3''$.
\end{lem}

To obtain a second equation between $\alpha_2$ and~$\alpha_5$, we first
need to determine the probability of solubility given a refinement of
the triple line configuration.  For $j=1,2$, let $\nu_j$ denote the
probability of solubility for cubics~$C$ whose reduction is the triple
line $Y=0$ and that also satisfy the condition that the coefficient
of~$X^3$ has valuation exactly~$j$ and the coefficients of $X^2Y$,
$X^2Z$ have valuations at least~$j$ (in earlier notation: $c_3$ has a
triple root modulo~$p$, \,$v(c_0)=j$, $v(c_1)\ge j$, and $v(c_2)\ge 1$).
\begin{lem}
\label{lem:nu}
\[
  \nu_1 = \frac{2p^8 + p^6 - 3p^5 + 3p^4 - p^2 - 2}{3(p^8-1)};\qquad
  \nu_2 = \frac{3p^8 - 3p^7 + 3p^6 - p^4 - 2}{3(p^8-1)}.
\]
\end{lem}
\begin{proof}
%  Recall (Corollary~\ref{cor:sigmatau}) that 
  Let $\sigma$ and $\tau$ denote the probabilities that a binary cubic
  over $\F_p$ has a simple root over $\F_p$ or a triple root,
  respectively, as in Corollary~\ref{cor:sigmatau}.  The corresponding
  probabilities for monic cubic polynomials are denoted $\sigma_1$ and
  $\tau_1$, respectively, as in Corollary~\ref{cor:sigma1tau1}.

  We arrange the~$10$ coefficients of the ternary cubic form~$C$ in a
  triangle with the $Z^3$-coefficient at the top, the
  $X^3$-coefficient at bottom left and $Y^3$-coefficient at bottom
  right, and indicate their valuations using the same
  equality/inequality notation as before.  The condition on a member $C$ in the set 
  %in the statement 
  of cubics whose density is measured by $\nu_1$ is then expressed by
\[
\begin{matrix}
Z^3 & \ge1 &      &      &      & \\
    & \ge1 & \ge1 &      &      & \\
    & \ge1 & \ge1 & \ge1 &      & \\
X^3 & =1   & \ge1 & \ge1 & =0   & Y^3
\end{matrix}
\]
Any solution must have $y\equiv0\pmod{p}$, so we substitute $pY$
for~$Y$ and divide by~$p$ to obtain
\[
\begin{matrix}
Z^3 & \ge0 &      &      &      & \\
    & \ge0 & \ge1 &      &      & \\
    & \ge0 & \ge1 & \ge2 &      & \\
X^3 &   =0 & \ge1 & \ge2 & =2   & Y^3
\end{matrix}
\]
Now the reduction is a binary cubic in~$X,Z$ with unit $X^3$
coefficient.  If it has a simple root in~$\F_p$
(probability~$\sigma_1$), it lifts to a $p$-adic root and~$C$ is
soluble with $y=0$.  Otherwise, $C$ is insoluble unless there is a
triple root (probability~$\tau_1$), since otherwise we could force
$x\equiv z\equiv0\pmod{p}$.  Given a triple root, we can shift the
root to~$0$, so that the binary cubic is a constant times~$X^3$.  Now
the valuations are
\[
\begin{matrix}
Z^3 & \ge1 &      &      &      & \\
    & \ge1 & \ge1 &      &      & \\
    & \ge1 & \ge1 & \ge2 &      & \\
X^3 &   =0 & \ge1 & \ge2 & =2   & Y^3
\end{matrix}
\]

The probability of solubility in this case is $\nu_2$. Tracing through
the arguments so far we have $\nu_1 = \sigma_1 + \tau_1 \nu_2$.  We
replace $X$ by $pX$ and divide through by~$p$:
\[
\begin{matrix}
Z^3 & \ge0 &      &      &      & \\
    & \ge1 & \ge0 &      &      & \\
    & \ge2 & \ge1 & \ge1 &      & \\
X^3 &   =2 & \ge2 & \ge2 & =1   & Y^3
\end{matrix}
\]
If the coefficient of $YZ^2$ has valuation~$0$ (with
probability~$1-1/p$), then the reduction is a binary cubic in~$Y,Z$
with a simple root, so $C$ is soluble.  Otherwise, $C$ reduces to a
multiple of~$Z^3$, but since $z\equiv0\pmod{p}$ is not allowed,
solubility requires that the coefficient of~$Z^3$ also has valuation
at least~$1$.  So we have insolubility with probability~$1/p-1/p^2$
and otherwise we can divide through by~$p$ to obtain
\[
\begin{matrix}
Z^3 & \ge0 &      &      &      & \\
    & \ge0 & \ge0 &      &      & \\
    & \ge1 & \ge0 & \ge0 &      & \\
X^3 &   =1 & \ge1 & \ge1 & =0   & Y^3
\end{matrix}
\]
If the coefficients of $XYZ$ and $XZ^2$ are not both zero modulo~$p$ (probability
$1-1/p^2$), then we have solubility since the reduction is either
absolutely irreducible or has a simple linear factor over~$\F_p$.
Otherwise (probability $1/p^2$) we have
\[
\begin{matrix}
Z^3 & \ge0 &      &      &      & \\
    & \ge1 & \ge0 &      &      & \\
    & \ge1 & \ge1 & \ge0 &      & \\
X^3 &   =1 & \ge1 & \ge1 & =0   & Y^3
\end{matrix}
\]
where the reduction is a binary cubic in $Y,Z$ with unit $Y^3$
coefficient.  We have solubility if it has a simple $\F_p$-root
(probability~$\sigma_1$), otherwise insolubility unless it has a
triple root (probability~$\tau_1$), in which case the valuations are
exactly as at the start, where solubility has probability~$\nu_1$.

Tracing through the above, we see that
\[
\nu_2 = (1-1/p)\cdot1 +
 (1/p^2)\cdot\bigl((1-1/p^2)\cdot1 + (1/p^2)\cdot(\sigma_1 + \tau_1\nu_1)\bigr).
\]
Recalling that $\nu_1 = \sigma_1 + \tau_1 \nu_2$, we may now solve
for both~$\nu_1$ and~$\nu_2$.
\end{proof}

The second equation linking $\alpha_5$ and $\alpha_2$ will involve the
above quantities~$\nu_1$ and $\nu_2$, and uses an argument similar to the
one used in the star case.

\begin{lem}\label{lem:eq2}
$ \alpha_2 = \sigma + \tau\nu_2 + (1/p^4) \bigl((1-1/p^3)
 + (1/p^3) (\sigma + \tau\nu_1 + (1/p^4)\alpha_5)\bigr).$
\end{lem}
\begin{proof}
  We may assume that the triple line modulo~$p$ is $X=0$.  
Using the same notation as before, let us write 
\[
   C = c_0X^3 + c_1(Y,Z)X^2 + c_2(Y,Z)X + c_3(Y,Z)
\]
where each $c_j$ is a binary form of degree~$j$; then the valuations of the
$c_j$ satisfy
\[
     {}=0 \qquad {}\ge1 \qquad {}\ge1 \qquad {}\ge1.
\]
Solutions must have $x\equiv0\pmod{p}$ so we replace $X$ by $pX$ and
divide by~$p$ to obtain
\[
     {}=2 \qquad {}\ge2 \qquad {}\ge1 \qquad {}\ge0,
\]
so that $C$ reduces to a binary cubic in $Y,Z$.

With probability~$\sigma$ this binary cubic has a simple $\F_p$-root
and~$C$ is soluble.  With probability~$\tau$ it has a triple root (without loss 
of generality the reduction is~$Z^3$); in this case,
$C$ is soluble with probability~$\nu_2$.  Otherwise, for solubility we
require $v(c_3)\ge1$, as otherwise solutions would have $y\equiv
z\equiv0\pmod{p}$, which is not allowed since we have already
scaled~$X$.  So with probability $1/p^4$ we have the valuations satisfying
\[
     {}=2 \qquad {}\ge2 \qquad {}\ge1 \qquad {}\ge1
\]
and dividing through by~$p$ we obtain
\[
     {}=1 \qquad {}\ge1 \qquad {}\ge0 \qquad {}\ge0.
\]
Now, with probability $1-1/p^3$ we have $v(c_2)=0$ and solubility
since %there is a
the reduction is either absolutely irreducible or has a simple linear
factor over~$\F_p$.  Otherwise (probability $1/p^3$), we have
$v(c_2)\ge1$:
\[
     {}=1 \qquad {}\ge1 \qquad {}\ge1 \qquad {}\ge0.
\]
As at the start, we have solubility if $c_3$ has a simple $\F_p$-root
(probability~$\sigma$), solubility with probability~$\nu_1$ if it has a
triple root (probability~$\tau$), and otherwise
for solubility we require $v(c_3)\ge1$ (probability~$1/p^4$) in which
case we divide through by~$p$.  The latter gives an arbitrary cubic subject
to the point condition, where the probability of solubility
is~$\alpha_5$.

Tracing through the above steps, we see that % $\alpha_5$ satisfies
\[
 \alpha_2 = \sigma\cdot1 + \tau\nu_2 + (1/p^4) \left((1-1/p^3)\cdot1
 + (1/p^3)\cdot(\sigma\cdot1 + \tau\nu_1 + (1/p^4)\alpha_5)\right),
\]
which is the equation stated.
\end{proof}

Using the equations given in Lemmas~\ref{lem:eq1} and~\ref{lem:eq2},
together with the known value of~$\alpha_1$, we can solve for
$\alpha_2$, obtaining the following.
\begin{prop}
  \[ \alpha_2 = 1 - \frac{p^{14} + 3 p^{11} + p^8 + 2 p^7 + p^5 + p^4
    + 1} { 3(p^2+1)(p^2+p+1)(p^4+1)(p^6+p^3+1) }. \]
\end{prop}

\subsection{Conclusion}

The probability of insolubility of a ternary cubic with coefficients
in~$\Q_p$ is therefore
\begin{align*}
  1-\rho(p)  &=    \frac{p^{10}}{p^{10}-1}
  (\beta_1(1-\alpha_1) + \beta_2(1-\alpha_2) + \beta_3) \\
             &=   f(p)/g(p),
\end{align*}
where $f$ and $g$ are the polynomials given in the statement of
Theorem~\ref{mainlocal}.

\section{The probability that a random plane cubic over $\Q_p$ has points over the maximal unramified extension of $\Q_p$}\label{sec:unr}

Let $\Q_p^{\nr}$ denote the maximal unramified extension of $\Q_p$. 
 
\begin{thm}\label{mainlocalnr}
The probability that a random plane cubic curve over $\Q_p$ has a
$\Q_p^{\nr}$-rational point is given by
\[ \rho^{\nr}(p) = 1-\frac{p^{11}(p-1)(p^2-1)(p^3-1)}{(p^8-1)(p^9-1)(p^{10}-1)}.
\]
\end{thm} 
Note that the formula for the probability in Theorem~\ref{mainlocalnr} is much simpler than
that in Theorem~\ref{mainlocal}. 

The proof of Theorem~\ref{mainlocal} may be regarded as an algorithm for 
testing whether a plane cubic is locally soluble (i.e., has a $\Q_p$-point), where 
we are able to determine explicitly the probability of entering each step of the algorithm.
The algorithm terminates either when there is a smooth $\F_p$-point on the
reduction (in which case it lifts to a $\Q_p$-point by Hensel's lemma) or when
we reach one of the following four situations:

\begin{description}
\item[(I$_{3m}$)] The reduction $\overline{C}$ consists of three non-concurrent
lines conjugate over $\F_{p^3}$ (the triangle configuration). 
\item[(IV)] The cubic is $\GL_3(\Z_p)$-equivalent to a cubic of the form
\[
   C = c_0X^3 + c_1(Y,Z)X^2 + c_2(Y,Z)X + c_3(Y,Z)
\]
where each $c_j$ is a binary form of degree~$j$ and the
valuations of the $c_j$ satisfy
\[
     {}=1 \qquad {}\ge1 \qquad {}\ge1 \qquad {}=0
\]
with $c_3$ irreducible modulo~$p$.
\item[(IV$^*$)] As in (IV), except that the valuations satisfy
\[
     {}=2 \qquad {}\ge2 \qquad {}\ge1 \qquad {}=0.
\]
\item[(Cr)] The cubic is $\GL_3(\Z_p)$-equivalent to a cubic $C$ 
whose coefficients have valuations satisfying
\[
\begin{matrix}
Z^3 & =2 &      &      &      & \\
    & \ge2 & \ge2 &      &      & \\
    & \ge1 & \ge1 & \ge2 &      & \\
X^3 & =0   & \ge1 & \ge1 & =1   & Y^3
\end{matrix}
\]
\end{description}

In the first three cases the cubic is insoluble over $\Q_p$ but soluble
over $\Q_p^{\nr}$; it can be shown that the Jacobian is an elliptic curve
$E/\Q_p$ with Kodaira symbol I$_{3m}$, IV, or IV$^*$ as indicated, and 
Tamagawa number divisible by $3$. 
In the final case, the cubic is a critical model, in the terminology of
\cite[Definition~5.1]{CFS}. As noted in \cite{CFS}, critical models
are insoluble over $\Q_p^{\nr}$. 

The probability that a cubic is insoluble over $\Q_p$, denoted $f(p)/g(p)$
in Theorem~\ref{mainlocal}, may thus be written as a sum of four
terms, corresponding to the four situations above. This allows us to naturally adapt the proof of Theorem~\ref{mainlocal} to a proof also of Theorem~\ref{mainlocalnr}.

\begin{ProofOf}{Theorem~\ref{mainlocalnr}} Since the proof of 
Theorem~\ref{mainlocalnr} is similar to that of Theorem~\ref{mainlocal},
we only highlight the differences. Let $\alpha'_1$, $\alpha'_2$,
$\nu'_1$, $\nu'_2$ be the probabilities of insolubility over 
$\Q_p^{\nr}$, in the situations where we earlier wrote $\alpha_1$, $\alpha_2$,
$\nu_1$, $\nu_2$ for the probabilities of solubility over $\Q_p$. 
The analogue of Lemma~\ref{lem:nu} gives
\[ \nu'_1 \,=\, \frac{p^4(p-1)}{p^8-1}, \qquad \nu'_2 \,=\, \frac{p^6(p-1)}{p^8-1}. \]
We have $\alpha'_1 = 0$. The analogues of Lemmas~\ref{lem:eq1} 
and~\ref{lem:eq2} give
\[ \alpha'_2 \,=\, \tau \nu'_2 + \frac1{p^7}\left(\tau \nu'_1 + \frac1{p^4}
\beta''_2 \alpha'_2\right). \]
Substituting $\nu_1' = \nu_2'/p^2$ and solving for $\alpha'_2$ gives
\[ \alpha'_2 \,=\, \frac{p^5(p^2-1)}{p^9-1} \nu'_2 \,=\, 
\frac{p^{11}(p-1)(p^2-1)}{(p^8-1)(p^9-1)}. \]
Finally, the probability of insolubility over $\Q_p^{\nr}$ is
\[ \frac{p^{10}}{p^{10}-1}\beta_2 \alpha'_2 \,=\, 
\frac{p^{11}(p-1)(p^2-1)(p^3-1)}{(p^8-1)(p^9-1)(p^{10}-1)}.\]
\end{ProofOf}

By the same argument as in \cite{PV}, we see that Theorem~\ref{mainlocalnr} implies that the probability $\rho^{\nr}$ that a random plane cubic curve over $\Q$ has a point over $\Q_p^{\nr}$ for all primes $p$ is given by 
$$\rho^{\nr}=\prod_p \rho^{\nr}(p) = \prod_p\left(1-\frac{p^{11}(p-1)(p^2-1)(p^3-1)}{(p^8-1)(p^9-1)(p^{10}-1)} \right)\approx
99.96676\%.$$
By \cite[Thm.\ 3.5]{CFS}, this may be interpreted as the probability that a
plane cubic curve over $\Q$ has minimal discriminant the same as that of
its Jacobian elliptic curve.

\subsection*{Acknowledgments}
We thank Michael Stoll and Damiano Testa for helpful conversations.
The first author was supported by a Simons Investigator Grant and NSF
grant~DMS-1001828.

\end{document}